\documentclass[12pt,reqno]{amsart}
\usepackage{mathrsfs,amsmath,amssymb,latexsym,amsthm,amsfonts, fullpage}
\usepackage[usenames, dvipsnames]{color}

 \usepackage[bookmarksnumbered,plainpages,backref]{hyperref}
\usepackage[utf8]{inputenc}
\usepackage[normalem]{ulem}

\title{The Baouendi-Treves approximation theorem for Gevrey classes and applications}

\author {Gustavo Hoepfner}
\address{Departamento de Matem\'atica, Universidade Federal de S\~ao Carlos, S\~ao Carlos, SP, 13565-905, Brasil}
\email{hoepfner@dm.ufscar.br}

\author {Renan D. Medrado}
\address{ Instituto de Matemática, Universidade Federal de Alagoas, Maceió, AL,   57072-970, Brasil}
\email{renan.medrado@im.ufal.br}

\author {Luis F. Ragognette}
\address{Departamento de Matem\'atica, Universidade Federal de S\~ao Carlos, S\~ao Carlos, SP, 13565-905, Brasil}
\email{luisragognette@dm.ufscar.br}

\date{}

\thanks{The first author was partially supported by FAPESP (2017/03825-1 and 2017/06993-2) and CNPq (305746/2015-4). The third author was partially supported by FAPESP (2016/13620-5 and 2017/13450-5).}

\theoremstyle{definition} 

\newtheorem{Def}{Definition}[section]

\theoremstyle{definition}

\newtheorem{remark}[Def]{Remark}

\theoremstyle{plain}
\newtheorem{Pro}[Def]{Proposition}

\newtheorem{Theo}[Def]{Theorem}
\newtheorem{Lem}[Def]{Lemma}

\newcommand{\XX}{\mathrm{X}}
\newcommand{\MM}{\mathrm{M}}
\newcommand{\LL}{\mathrm{L}}

\newcommand{\RN}{\R^N}

\newcommand{\Rm}{\R^m}

\newcommand{\ZN}{\Z^N}

\newcommand{\supp}{\mathrm{supp}}

\newcommand{\R}{\mathbb{R}}

\newcommand{\Rn}{\R^{n}}
\newcommand{\C}{\mathbb{C}}

\newcommand{\Z}{\mathbb{Z}}

\newcommand{\Cinf}{C^{\infty}}

\newcommand{\lra}{\longrightarrow}

\newcommand{\del}{\partial}

\renewcommand{\Re}{\mathrm{Re}\,}

\newcommand{\dd}{\textnormal{d}}

\numberwithin{equation}{section}

\begin{document}

\begin{abstract}
In this work we show how to extend the seminal Baouendi-Treves approximation theorem 
for Gevrey functions and ultradistributions.  As  applications we present a Gevrey version of the approximate Poincaré Lemma and study ultradistributions vanishing on maximally real submanifolds. 
\end{abstract}

\maketitle

\setcounter{tocdepth}{1}

\tableofcontents

\section{Introduction}

The goal of this paper is to  extend the celebrated Baouendi-Treves approximation theorem  to Gevrey functions and ultradistributions. The classical Baouendi-Treves theorem has deep implications in the theory of CR geometry and in the theory of local solvability of locally integrable structures.

Let us denote by $\Omega$ an open subset of $\R^N$. A locally integrable structure is a  subbundle $\mathcal L$ of the complexified tangent bundle $\C T\Omega$  if given an arbitrary point $p_0\in\Omega$ there are an open neighborhood $U_0$ of $p_0$  and functions $Z_1,\ldots,Z_m \in C^\infty(U_0)$  such that the orthogonal of $\mathcal{L}$ is generated over $U_0$ by their differentials $\dd Z_1,\ldots,\dd Z_m$. We say that $u$ is a  solution of $\mathcal L$ if, for every (smooth) local section $\LL\in\mathcal L$, we have $\LL u=0.$
The Baouendi-Treves approximation theorem states that  any  $u$ in $C^k(\Omega)$, $k\in\{0,1,2,\dots,\infty\}$,  that is  solution of $\mathcal{L}$ can be approximated in  a small neighborhood of any given point of $\Omega$ in the $C^k$-topology by polynomials in $Z_1,\dots,Z_m$ and if $u\in \mathcal{D}'(\Omega)$ is a solution a similar result holds in the topology of $\mathcal{D}'$. Further generalizations for  Lebesgue spaces
$L^p$, $1\le p<\infty$; Sobolev spaces; Hölder spaces; and (localizable) Hardy spaces $h^p$, $0<p<\infty$ were given in \cite{HM98, bch_iis}.

Our main theorem extends to the class of Gevrey  functions and their dual spaces.

\begin{Theo}[Baouendi-Treves approximation formula] \label{BTAPPthm}
Let $\mathcal L$ be a $G^{s}-$locally integrable structure on $\Omega.$ Let us assume that there is $Z=(Z_1,\dots,Z_m): \Omega \lra \C^m$ of class $G^{s}$ such that  $\dd Z_1,\dots,\dd Z_m$ spans $ \mathcal L^\perp$ over $\Omega$. Then, for any $p\in \Omega$, there exist two open sets $U$ and $W$ with $\overline U\subset W\subset \Omega$ such that
\begin{enumerate}
\item any $u\in G^{s} (W)$ that is a solution of $\mathcal{L}$ in $W$ is the limit in  $ G^{s} (U)$ of a sequence of polynomial solutions $P_j(Z)$.
\item any $u\in\mathcal {D}'_s (W)$ that is a solution of $\mathcal{L}$ in $W$ is the limit in  $\mathcal{D}'_s (U)$ of a sequence of polynomial solutions $P_j(Z)$. 
\end{enumerate} 
\end{Theo}

The Baoeundi-Treves approximation formula was already proved in the special case when the Gevrey locally integrable structure has corank zero, i.e., every point has a neighborhood $U$ where we have defined $Z_1, \ldots, Z_N$ Gevrey functions whose differential generate $\C T^{\ast} \Omega$ (see \cite{c00} and \cite{lfr2019}).

Our first application is a Gevrey version of a result called approximate Poincaré Lemma (see \cite{treves_has}). It is a useful lemma in the theory of local solvability of locally integrable structures that essentially says that a form that is $\mathbb{L}$-closed is the limit of $\mathbb{L}$-exact forms, here $\mathbb{L}$ is a differential operator induced by the de Rham operator.

Another application says that an ultradistribution solution of $\mathcal{L}$ that vanishes in a submanifold maximally real  with respect to $\mathcal{L}$ must be zero in a neighborhood of the submanifold.

The proof of Theorem~\ref{BTAPPthm} will be divided in two steps: first for ultradifferentiable  functions (classical solutions) and  second for ultradistributions (weak solutions).

The novelty here is, in one hand, to provide a finer  way to write the commutator formula first given by \cite{BT81,bch_iis}, see \eqref{XalphaGtau}, which allow us to obtain optimal  control on the constants appearing in the process of differentiating indefinitely the approximation operators when the solution is classic. On the other hand, when the solution is only an ultradistribution, we need to justify the approximation operator by proving that the traces of solutions of $\mathcal L$ are well defined (by the same formula given in \cite[Section 8]{Hor90a} in the case of a distribution)  and then take the full advantage of this formula \eqref{TRACE} to obtain the approximation scheme in these case. 

We point out that the original arguments for weak solutions cannot be applied in our situation since ultradistributions cannot be represented by a finite order differential operator. However our argument can be used to recover the original result without making use of either: the representation of distributions by means of a finite order partial differential operators, or Sobolev embedding theorems. Thus we strongly believe that our arguments can be used to simplify the original arguments.

The paper is organized as follows: in Section \ref{S2} we recall some definitions and basic results of the Gevrey functions and introduce the locally integrable structures.  The proof of Theorem \ref{BTAPPthm} is presented in Section \ref{DBTG}, first for Gevrey functions, Subsection~\ref{ap}, then for Gevrey ultradistributions, Subsection~\ref{ap2}. We present two main applications: the first  is given  in Section~\ref{AAPL} where we use Theorem \ref{BTAPPthm} to prove the approximate Poincaré Lemma in the Gevrey topology; and the second is treated in Section~\ref{AUVOMRS} where  we study ultradistributions vanishing on maximally real submanifolds.
The approximation theorem for more general classes of ultradifferentiable functions and ultradistributions is discussed in Section~\ref{BTDC}.
Finally, we conclude with two sections in the Appendix regarding some technicalities needed troughout the paper.

\section{Definitions and Preliminar Results}\label{S2}

Let $\Omega$ be an open subset of $\RN$ and fix $s\geq 1$. A Gevrey function of order $s$ in $\Omega$ is a smooth function $f\in \Cinf(\Omega)$ such that for every $K$ compact subset of $\Omega$ there is a $h>0$ such that
\begin{align}
  \|f\|_{h, K}\doteq \sup_{\alpha\in Z_+^{N}} \Big(\frac{1}{h^{|\alpha|}\alpha!^s} \sup_{x\in K} |\del^{\alpha} f(x)|\Big)<\infty.
\end{align}
We will denote the space of Gevrey functions of order $s$ in $\Omega$ by $G^s(\Omega)$. We recall that $G^{1}(\Omega)$ is the space of real-analytic functions in $\Omega$. In this work we will always assume that $s>1$ and we shall denote by $G^{s}_c(\Omega)$ the space of Gevrey functions of order $s$ with compact support.

If $V\subset \subset \Omega$ and $h>0$, we shall denote by $G^{s,h}(\overline{V})$ the space of all smooth functions $f$ in $\overline{V}$ for which
$  \|f\|_{h, \overline{V}}<\infty.$ Moreover, we denote by $G^{s}(\overline{V})$ the space of all smooth functions $f$ in $\overline{V}$ for which there is a $h>0$ such that $f\in G^{s,h}(\overline{V})$ and we denote by $G^{s,h}_c(\overline{V})$ the space of all $f\in \Cinf(\RN)$ with support in $\overline{V}$ such that $f|_{\overline{V}}\in  G^{s,h}(\overline{V})$.

 The topological dual of $G^{s}_c(\Omega)$ will be called the space of ultradistributions of order $s$ and  will be denote by $\mathcal{D}'_s(\Omega)$. The continuity of $u\in \mathcal{D}'_s(\Omega)$ can be expressed in the following way: for every $V\subset \subset \Omega$  and every $h>0$ there is $C_h>0$ such that
\begin{align*}
  |u(\varphi)|\leq C_h \|\varphi\|_{h, \overline{V}}, 
\end{align*}
for every $\varphi\in G^{s,h}_c(\overline{V})$. We will denote by $\mathcal{E}'_s(\Omega)$ the space of ultradistributions with compact support in $\Omega$.

Let us assume that $0$ belongs to $\Omega$, $N=m+n$ and consider a $G^{s}$-locally integrable structure of rank $n$ in $\Omega$, i.e.,  a subbundle $\mathcal L$ of $\mathbb C T\Omega$ of rank $n$ over $\Omega$ which the orthogonal, $\mathcal{L}^{\perp},$  is locally generated by the differentials of $m$ Gevrey functions of order $s$ in $\Omega$.

According to \cite{bch_iis,treves_has}\footnote{It is easy to check that their proofs also work in the $G^s$-category.} we can assume, shrinking $\Omega$ around $0$ if necessary, the existence of a local system of $G^s$ coordinates $(x,t)=(x_1,\dots,x_m,t_1,\dots,t_n)$ in $\Omega$
 as well as a  map $\phi:\Omega\to\mathbb R^m$, $\phi=(\phi_1,\dots,\phi_m)$ of class $G^{s}$ 
satisfying 
\begin{equation}\label{eq z2}
\phi_k(0,0) =0, \  \dd_x \phi_k(0,0)=0,\quad k=1,\quad \ldots, m  
\end{equation}
such that
\begin{equation}\label{eq z1}
Z_k(x,t) = x_k +i\phi_k(x,t), \quad k=1,\dots,m.
\end{equation}

Denote by $B^{\R^p}_R(0)= \{ x \in \R^p: |x|<R\}$ and define $V:=B_R^{\R^m}(0)\times B_R^{\R^n}(0)$. Since $\dd_x \phi_1(0)=\cdots \dd_x \phi_m(0)=0$, we can choose a positive number $R$   such that $V\subset\subset \Omega$ and
\begin{equation}\label{eq z4}
|\phi_k(x,t)-\phi_k(y,t)| \le \tfrac12 |x-y|, \quad \forall (x,t), (y,t)\in \overline{V}.
\end{equation}

Fix  an open neighborhood $W\subset \subset \Omega$ of $\overline{V}.$ Modifying the imaginary part of $Z$ outside of $W$ using cutoff functions of class $G^s$ we can obtain a locally integrable structure defined globally in $\RN$ that agrees with $\mathcal L$ in $W$. Abusing of notation we will still denote this new structure by $\mathcal{L}$ and assume that \eqref{eq z4} holds globally in $\RN$. Note that the conclusions that we will obtain for this new structure $\mathcal L$ will also be true for the old structure in $V$.

Since $\dd Z_1, \ldots, \dd Z_m, \dd t_1, \ldots, \dd t_n$ is a global frame for $\C T^{\ast} \RN$ we can consider its dual frame in $\C T \RN$, i.e., consider $N$ vector fields:
\begin{align}\label{LseMs}
  \MM_1, \ldots, \MM_m, \LL_1, \ldots, \LL_n
\end{align}
with the property that
\begin{align*}
  \dd Z_{k}( \MM_{k'})= \delta_{k k'}, \quad  \dd Z_{k}( L_{j})= 0, \quad\quad k, k'\in \{1, \ldots, m\}, \quad j \in \{1, \ldots, n\},\\
  \dd t_{j}( \MM_{k})= 0, \quad  \quad \dd t_{j}( L_{j'})= \delta_{j j'}, \quad \quad k \in\{ 1, \ldots, m\},\quad j,j' \in \{1, \ldots, n\}.
\end{align*}

Finally, note that the differential of any $C^1$ function $w(x,t)$ can be expressed in the basis $\{\dd Z_1, \dots , \dd Z_m, \dd t_1,\dots, \dd t_n\}$ of $\mathbb C T^*\mathbb R^N$ as
\begin{equation}
\dd w = \sum_{j=1}^n \LL_j w\, \dd t_j   + \sum_{k=1}^m\MM_k w\,\dd Z_k.
\end{equation}

Let $\XX_1, \ldots, \XX_N$ be a family of $N$ pairwising commuting smooth vector fields  that form a global frame to $\C T \Omega$. We can define the space of Gevrey functions regarding $\XX_1, \ldots, \XX_N$  as the space of all $f\in \Cinf(\Omega)$ such that for every $K$ compact subset of $\Omega$ there is a $h>0$ such that
\begin{align}
  \sup_{\alpha\in Z_+^{N}} \Big(\frac{1}{h^{|\alpha|}\alpha!^s} \sup_{(x,t)\in K} |\XX^{\alpha} f(x,t)|\Big)<\infty.
\end{align}
We shall denote this space by $G^{s}(\Omega; \XX)$. A sequence of functions $f_\nu\in G^{s}(\Omega; \XX)$ converges to $f\in G^s(\Omega; \XX)$ if for every $K\subset \Omega$ compact there is $h>0$ such that for every $\epsilon>0$ there is $\nu_0$ such that
\begin{align*}
  \sup_{\alpha\in Z_+^{N}} \Big(\frac{1}{h^{|\alpha|}\alpha!^s} \sup_{(x,t)\in K} |\XX^{\alpha}f_{\nu}(x,t)- \XX^{\alpha}f(x,t)|\Big)< \epsilon,
\end{align*}
for every $\nu>\nu_0$.

Analogously, we denote by $G^s(\Omega; \LL, \MM)$ the space of Gevrey functions with respect to the vector fields  considered in \eqref{LseMs}, associated with a locally integrable structure. Since $\mathcal{L}$ is a $G^{s}$-locally integrable structure, it was proved in \cite{lfr2019} that
\begin{equation}\label{isomorphic}
G^s(\Omega; \LL, \MM)\  \text{ is isomorphic to } \ G^s(\Omega) \ \text{ as topological spaces} 
\end{equation}
and the same holds for compact sets. These spaces will play an important role in this work since part of the proof will be to show that a sequence of functions converges in $G^{s}(\overline{V}; \MM, \LL)$ (and consequently in $G^{s}(\overline{V})$) for a relatively compact open subset $V$ of $\Omega$.
 We will also use the following notation: for every $k$ positive integer we denote
 \begin{align*}
   \|f\|_{C^{k}(\overline{V})}= \sum_{|\alpha|\leq k} \sup_{x\in \overline{V}}|\del^{\alpha}f(x)|
 \end{align*}
 where $f \in C^{k}(\overline{V})$.



\section{The ultradifferentiable Baouendi-Treves approximation formula \label{DBTG}}

In this section we will present the Baouendi-Treves approximation formula for ultradifferentiable functions and ultradistributions that are solutions of a locally integrable structure of arbitrary rank. It is easy to see that the theorem follows if we prove that the solutions are limit in the appropriate topology of entire functions in $Z$.

%





\subsection{Proof of Baouendi-Treves  approximation theorem  in $G^{s}$}\label{ap}

Let  $u\in G^s(\Omega)$ be a solution of $\mathcal L$ in $W$. For each $\chi\in G_c^{s}(B_{R}^{\R^m}(0))$  and, 
for each $\tau>0$, define the function $E_\tau^{\chi}[ u]$ by
\begin{equation}
E_\tau^{\chi}[ u](x,t):=\Big(\frac\tau{\pi}\Big)^{\tfrac m2}\!\! \int_{\R^m}  e^{-\tau \langle Z(x,t)-Z(y,0) \rangle^2} \chi(y)u(y,0) \,{\rm det} Z_x(y,0)\,\dd y, \quad (x,t)\in\R^N.
\end{equation}
For each $\tau>0$, $E_\tau^{\chi}[u]$ is an entire function of $Z(x,t).$ Thus $E_\tau^{\chi}[ u]\in G^s(\RN)$ and  is a solution of $\mathcal L$.
Consider also the functions  defined by
\begin{equation}\label{eq Gtau}
G_\tau^{\chi}[ u](x,t): = \Big(\frac\tau{\pi}\Big)^{\tfrac m2}\!\! \int_{\R^m}  e^{-\tau \langle Z(x,t)-Z(y,t)\rangle^2} \chi(y)u(y,t) \, {\rm det} Z_x(y,t)\,\dd y, 
\end{equation}
and,
\begin{equation}\label{eq Rtau}
R_\tau^{\chi}[ u](x,t):=G_\tau^{\chi}[ u](x,t) - E_\tau^{\chi}[ u](x,t).
\end{equation}

 We note that $G_\tau^{\chi}[u]$  converges to $\chi u$ even when $u$ is not a solution of $\mathcal{L}$. 
 
\begin{Pro}\label{GtauFuncao}
Let $\chi\in G^{s}_c(B^{\R^m}_{R}(0))$ and $u \in G^{s}(\overline{V})$. Then $G^{\chi}_\tau[u]$ converges to $\chi u$ in $G^{s}(\overline{V})$ when $\tau \lra \infty$. 
\end{Pro}
\begin{proof}
It is enough to prove (see \eqref{isomorphic}) that
\begin{equation}\label{toshow}
G_\tau^{\chi}[ u](x,t) \to \chi u \quad \text{in }\quad  G^{s}(\overline{V}; \MM,\LL).
\end{equation}
Note that we may write 
\begin{equation}\label{eq:spliting Gtau}
G_\tau^{\chi}[ u](x,t) - \chi(x)u(x,t) = I_{\tau}^{\chi}[ u](x,t) - J_{\tau}^{\chi}[ u](x,t)
\end{equation}
where $I_{\tau}^{\chi}[ u]$ and $J_{\tau}^{\chi}[ u]$ can be written, after the change of variables $y \mapsto x+ \tau^{-1/2} y$ in \eqref{eq Gtau}, as
\begin{align*}
  I_{\tau}^{\chi}[ u](x,t)&= \pi^{-\frac{ m}{2}}\!\! \int_{\R^m}  e^{-\langle Z_x(x,t)y \rangle^2}\big(v(x+ \tau^{-1/2}y, t)-  v(y,t)\big)\dd y,\\
  J_\tau^{\chi}[u](x,t)&= \pi^{-\frac{ m}{2}}\!\! \int_{\R^m}  \Big(e^{-\tau \langle Z(x,t)-Z(x+ \tau^{-1/2}y,t)\rangle^2}-e^{-\langle Z_x(x,t)y \rangle^{2}}\Big) v(x+ \tau^{-1/2}y,t)\,\dd y,
\end{align*}
and the function $v$ is defined by
\begin{align*}
  v(y,t)= \left\{
  \begin{array}{ll}
    \chi(y)u(y,t) \ {\rm det} Z_x(y,t),  \quad &(y,t)\in B^{\R^m}_R(0)\times B_R^{\R^n}(0),
    \\[4pt]
    0,  &  (y,t) \in(\R^{m}\setminus B^{\R^m}_R(0))\times B_R^{\R^n}(0).
  \end{array}\right.
\end{align*}
We have
\begin{align*}
  \big|v(x+ \tau^{-1/2}y, t)-  v(x,t)\big|&\leq \tau^{-\frac{1}{2}} \|\nabla v\|_{C(\overline{V})}\\[4pt]
  &\leq \tau^{-\frac{1}{2}}\|\chi\|_{C^{1}(B_{R}(0))} \|u\|_{C^{1}(\overline{V})} \|\det Z_x\|_{C^{1}(\overline{V})},
\end{align*}
therefore,
\begin{align}\label{def:Itau}
  |I_{\tau}^{\chi}[ u](x,t)|&\leq   \pi^{-\frac{ m}{2}}\!\! \int_{\R^m}  e^{- |y|^{2}+ |\phi_x(x,t) y|^{2}}\big|v(x+ \tau^{-1/2}y, t)-  v(x,t)\big|\dd y
\notag  \\[4pt]
  &\leq \tau^{-\frac{1}{2}}\frac{B}{\pi^{\frac{ m}{2}}}  \|\chi\|_{C^{1}(B_{R}(0))} \|u\|_{C^{1}(\overline{V})}   \int_{\R^m}  e^{- \frac{3}{4}|y|^{2}}\dd y,
\end{align}
where $B:=\|\det Z_x\|_{C^{1}(\overline{V})}$.

To estimate $J^{\chi}_\tau[u](x,t)$, we use the fact that $|e^{-\tau\langle Z(x,t)-Z(x+ \tau^{-1/2}y,t)\rangle^2}|\leq e^{-3|y|^{2}/4}$ and $|e^{-\langle Z_x(x,t)y \rangle^{2}}|\leq e^{-3|y|^{2}/4}$, to obtain
\begin{align}\label{Jtau2}
  |J_{\tau}^{\chi}[ u](x,t)|&\leq    \pi^{-\frac{ m}{2}} \|v\|_{C(\overline{V})} \int_{\R^m}  \Big|e^{-\tau\langle Z(x,t)-Z(x+ \tau^{-1/2}y,t) \rangle^2}-e^{-\langle Z_x(x,t)y \rangle^{2}}\Big|\,\dd y\notag\\[4pt]
                            &\leq    \pi^{-\frac{ m}{2}} \|v\|_{C(\overline{V})} \int_{|y|< A}  \Big|e^{-\tau \langle Z(x,t)-Z(x+ \tau^{-1/2}y,t)\rangle^2}-e^{-[Z_x(x,t)y]^{2}}\Big|\,\dd y\notag\\[4pt]
                            &\hskip.5cm+    \pi^{-\frac{ m}{2}} \|v\|_{C(\overline{V})} e^{-A^{2}/2} \int_{|y|\geq A}  2e^{- |y|^{2}/4}\,\dd y,
\end{align}
for every $A>0$.
To estimate the first integral on the rightmost hand-side of \eqref{Jtau2}, we fix $y$ and $t$ and note that $\zeta_1=[ Z(x,t)-Z(x+ \tau^{-1/2}y,t)]/\tau^{-1/2}$ converges to $\zeta_2=-Z_x(x,t)y$ uniformly in $x\in \R^m$ as $\tau$ goes to $\infty$ and so there is $C>0$ such that $|[\zeta_1]^{2}-[\zeta_2]^{2}|\leq C\tau^{-1/2}$. This implies that $\Re [\zeta_1]^{2}\geq 0$ and $\Re [\zeta_2]^{2}\geq 0$ and using that $e^{-\zeta}$ is a Lipschitz function on $\Re \zeta\geq 0$ we conclude that
\begin{align}\label{def:Jtau}
  |J_{\tau}^{\chi}[ u](x,t)|&\leq    \frac{B}{\pi^{\frac{ m}{2}}} \|\chi\|_{C(B_{R}(0))} \|u\|_{C(\overline{V})} \bigg( C A^{m} \tau^{-1/2}
                            +   e^{-A^{2}/2} \int_{|y|\geq A}  2e^{- |y|^{2}/4}\,\dd y \bigg).
\end{align}
Using  \eqref{eq:spliting Gtau}, we may rewrite Lemma II.1.4 and Lemma II.1.6 in  \cite{bch_iis} as
\begin{equation}\label{LeibnizM}
  \MM_kG^{\chi}_\tau[u]= G_\tau^{\chi} [\MM_ku]+ G_\tau^{\MM_k\chi}[u], \quad \forall k=1,\ldots, m
\end{equation}
and
\begin{equation}\label{LeibnizL}
  \LL_jG^{\chi}_\tau[u]= G_\tau^{\chi} [\LL_ju]+ G_\tau^{\LL_j\chi}[u], \quad \forall j=1, \ldots, n.
\end{equation}
In order to simplify the notation define $\XX_j= \LL_j$ for $j=1, \ldots, n$ and $\XX_{n+k}= \MM_k$, for $k=1,\dots,m$. Since $\XX_1, \ldots, \XX_{n+m}$ are pairwise commuting, we have  that
\begin{equation}
  \XX^{\alpha} G^{\chi}_\tau[u]= \sum_{\alpha'+\alpha''= \alpha}{ \alpha\choose \alpha'} G_\tau^{\XX^{\alpha'}\chi} [\XX^{\alpha''}u], \quad \forall \alpha \in \ZN_+.
\end{equation}
Thus it follows that
\begin{align}\label{XalphaGtau}
  \XX^{\alpha} G^{\chi}_\tau[u]- \XX^{\alpha}(\chi u)= \sum_{ \alpha'+\alpha''= \alpha } {\alpha \choose \alpha'} \big( G_\tau^{\XX^{\alpha'}\chi} [\XX^{\alpha''}u]-(\XX^{\alpha'}\chi) [\XX^{\alpha''}u]\big).
\end{align}
To estimate $\XX^{\alpha} G^{\chi}_\tau[u]- \XX^{\alpha}(\chi u)$ on $\overline{V}$ we will make use of \eqref{eq:spliting Gtau}, \eqref{def:Itau} and \eqref{def:Jtau}, together \eqref{XalphaGtau} to obtain
\begin{align}\label{des:inter0}
  |\XX^{\alpha} G^{\chi}_\tau[u](x,t)- \XX^{\alpha}(\chi u)(x,t)|&\leq  \sum_{ \alpha'+\alpha''= \alpha } {\alpha \choose \alpha'} \Big(|I_{\tau}^{\XX^{\alpha'}\chi}[\XX^{\alpha''} u](x,t)|+|J_{\tau}^{\XX^{\alpha'}\chi}[\XX^{\alpha''} u](x,t)| \Big)
\notag\\[4pt]
&\leq \tau^{-1/2}\frac{B \tilde{C}_A }{\pi^{\frac{ m}{2}}}  \sum_{\alpha'+\alpha''= \alpha} {\alpha \choose \alpha' } \|\XX^{\alpha'}\chi\|_{C^{1}(B_{R}(0))} \|\XX^{\alpha''}u\|_{C^{1}(\overline{V})}  
\notag\\[4pt]
&\hskip.4cm +e^{-A^{2}/2}\frac{B \hat{C}}{\pi^{\frac{ m}{2}}} \sum_{ \alpha'+\alpha''= \alpha }\!\!{ \alpha \choose \alpha'}\|\XX^{\alpha'}\chi\|_{C(B_{R}(0))} \|\XX^{\alpha''}u\|_{C(\overline{V})}   
\end{align}
where 
\[
\tilde{C}_A:=  \int  e^{- \frac{3}{4}|y|^{2}}\dd y+CA^{m}  \quad\text{and}\quad \hat{C}:= \int  2e^{- |y|^{2}/4}\,\dd y.
\]

Now we assume that $\chi \in G^{s}_c(B_{R}^{\R^m}(0))$ and $u \in G^{s}(\overline{V})$ thus, it follows from \eqref{isomorphic} that we can find $h>0$ such that for every $\alpha', \alpha''\in\ZN_+$, we have 
\begin{equation}\label{expectedinequalities}
\|\XX^{\alpha''}u\|_{C^{1}(\overline{V})}\leq h^{|\alpha''|+1} \|u\|_{h,\overline{V}} (|\alpha''|+1)!^{s} 
\ \ \text{and}\ \  \|\XX^{\alpha'}\chi\|_{C^{1}(B_{R}^{\R^m})}\leq h^{|\alpha'|+1} \|\chi\|_{h,\overline{B_{R}^{\R^m}(0)}} (|\alpha'|+1)!^{s}.
\end{equation}

We may use \eqref{des:inter0} and \eqref{expectedinequalities} to obtain
\begin{align}\label{des:inter1}
\displaystyle  \sup_{\overline{V}}\frac{|\XX^{\alpha} G^{\chi}_\tau[u]- \XX^{\alpha}(\chi u)|}{(2h)^{|\alpha|} |\alpha|!^{s}} &\leq \Big(\tau^{-1/2}\tilde{C}      +e^{-A^{2}/2} \hat{C}_A \Big) \frac{B A}{\pi^{\frac{ m}{2}}} h^{2}2^{3} \|\chi\|_{h,\overline{B_{R}(0)}}\|u\|_{h,\overline{V}},
\end{align}
where we used that $(|\alpha|+2)!^{s}\leq 2^{s (|\alpha|+3)} |\alpha|!^{s}$.
Now, for a given $\epsilon>0$ choose $A>0$ so that $e^{-A^{2}/2} \hat{C}\leq \epsilon/2$ and then choose $\tau>1$ so that $\tau^{-1/2}\tilde{C}_A\leq \epsilon/2$ to conclude that
$G^{\chi}_\tau[u]$ converges to $\chi u$ in $G^s(\overline{V}; \MM,\LL)$.
\end{proof}

We would  like to point out that the proof  yields a slightly stronger version of Proposition~\ref{GtauFuncao}.  Denote by $\mathcal{B}\big(G^{s}_c(B_R^{\R^m}(0))\times G^s(\overline{V}), G^s(\overline{V})\big)$ the space of the bilinear  continuous operator and denote by $P$ the bilinear operator defined by the usual product, i.e., $P(\chi, u)= \chi u$.
  \begin{Pro}
    The operator $G_\tau: G^s_c(B_R^{\R^m}(0))\times G^s(\overline{V}) \lra G^s(\overline{V})$ define by $G_\tau(\chi, u)= G_\tau^\chi[u]$ is a bilinear and continuous. Moreover, the sequence of operators $G_\tau$ converges to $P$ in $\mathcal{B}\big(G^{s}_c(B_R^{\R^m}(0))\times G^{s}(\overline{V}); G^{s}(\overline{V})\big)$ as $\tau \lra \infty$. 
  \end{Pro}

  Observe that if we take $\chi= 1$ in $B_{R/2}^{\R^m}(0)$ and define $U:= B_{S}^{\R^m}(0)\times B_{T}^{\R^n}(0)$, where $0<S\leq R/2$ and $0< T\leq R$, then $G^{\chi}_\tau[u]$ converges to $u$ in $G^{s}(U)$.

Next, we recall  (see, for instance,  \cite[pag. 59-60]{bch_iis}) that there exists a positive constant $T<R$ such that 
\begin{equation}\label{est:faseRtau}
\big|e^{-\tau \langle Z(x,t)-Z(y,t')\rangle^2}\big| \le e^{-\tau R^2/33}, 
\end{equation}
 for all $(x,t) \in B^{\mathbb{R}^m}_{R/4}(0)\times B^{\mathbb{R}^n}_{T}(0)\textrm{ and }  (y,t') \in \{y \in \R^m:|y|\ge R/2\} \times B_T^{\R^n}(0)$.
From now on, we fix the open set $U$ in the statement of Theorem~\ref{BTAPPthm} to be $B_{R/4}^{\R^m}(0)\times B_{T}^{\R^n}(0)$.
The proof of Theorem~\ref{BTAPPthm},  in $G^{s}$, will be complete once we proof the following result.

\begin{Pro}\label{RtauFuncao}
 Let $\chi\in G^{s}_c(B^{\R^m}_{R}(0))$, with $\chi=1$ in $B^{\R^m}_{R/2}(0)$  and $u \in G^{s}(\Omega)$ that is a solution of $\mathcal{L}$ in $W$. Then $R^{\chi}_\tau[u]$ converges to $0$ in $G^s(U)$ when $\tau \lra \infty$. 
\end{Pro}

\begin{proof}
It is a consequence of Stokes' theorem  that we  may write $R^{\chi}_\tau[ u]$ given in \eqref{eq Rtau} as
\begin{equation}
R^{\chi}_\tau[ u](x,t)= \Big(\frac\tau{\pi}\Big)^{\tfrac m2}\! \sum_{j=1}^n\int_{\R^m\times[0,t]}  e^{-\tau \langle Z(x,t)-Z(y,t')\rangle^2} (\LL_j\chi)(y,t')u(y,t')  \,{\rm det} Z_x(y,t')\,\dd t_j'\wedge \dd y.
\end{equation}
Since $\chi(y)=1$ for $|y|<R/2$ and ${\rm supp}\,\chi\subset B_R^{\R^m}(0)$, $\LL_j\chi$ vanishes for $\{|y|\le R/2\}\cup\{|y|\ge R\}$, 
 we can write
\begin{equation}\label{recall:Rtau}
R^{\chi}_\tau[ u](x,t)= \Big(\frac\tau{\pi}\Big)^{\tfrac m2}\! \sum_{j=1}^n\int_{A(\frac{R}{2},R)} \int_0^1  e^{-\tau\langle Z(x,t)-Z(y,r t)\rangle^2} (L_j\chi)(y,rt)u(y,rt)  \,{\rm det} Z_x(y,rt) t_j\dd r\dd y,
\end{equation}
where $A(\frac{R}{2},R):= \{y \in \mathbb{R}^m: \frac{R}{2} < |y| < R\}.$ For each $(\alpha, \beta)\in\Z^m_+\times\Z^n_+$ we may differentiate under the integration sign the expression in the right hand-side of \eqref{recall:Rtau} to obtain
\begin{align}\label{recall:RtauDer}
\partial_x^\alpha \partial_t^\beta &R^{\chi}_\tau[ u](x,t)
\\[6pt]
&= \Big(\frac\tau{\pi}\Big)^{\tfrac m2}\! \sum_{j=1}^n\int_{A(\frac{R}{2},R)}\int_0^1 \!\partial_x^\alpha\partial_t^\beta 
          \big\{e^{-\tau\langle Z(x,t)-Z(y,rt)\rangle^2} (L_j\chi)(y,rt)u(y,rt)  \,{\rm det} Z_x(y,rt)t_j\big\} \dd r\dd y 
\notag\\[6pt]
&= \Big(\frac\tau{\pi}\Big)^{\tfrac m2}\! \sum_{j=1}^n \sum_{\gamma\le\beta}{\beta \choose \gamma}\int_{A(\frac{R}{2},R)}\int_0^1 \Big\{\partial_x^\alpha\partial_t^\gamma \big\{e^{-\tau\langle Z(x,t)-Z(y,rt)\rangle^2}\big\} \times
\notag\\[6pt]
&\hskip6.5cm\times\partial_t^{\beta-\gamma}\big\{ (L_j\chi)(y,rt)u(y,rt) \, \,{\rm det} Z_x(y,rt)t_j\big\} \Big\}\dd r\dd y. \notag
\end{align}

We can use Lemma~\ref{lem:usandoFdB} with $f(y,r,x,t)=[Z(x,t)-Z(y,rt)]^2$ yielding that there are constants $C, h>0$ for which we have 
\begin{align}\label{est:usandoFdBa}
\displaystyle\left| \partial_x^\alpha\partial_t^\gamma\big\{e^{-\tau \langle Z(x,t)-Z(y,rt)\rangle^2}\big\} \right|
\leq C h^{|\alpha|+|\gamma|} (|\alpha|+|\gamma|)!^{s}  \,e^{-\tau\Re\langle Z(x,t)-Z(y,r t)\rangle^2+ s \tau^{1/s}}
\end{align}
for every $(x,y,t,r)\in B_{R/2}^{\R^m}(0)\times  A(\frac{R}{2},R)\times B_{T}^{\R^n}(0)\times [0,1].$

Since there exist $\tilde{h}>0$ so that
$(L_j\chi)\,u \det Z_x|_{\overline{V}}\in G^{s, \tilde{h}}(\overline{V})$ for each $j\in\{1,\dots,n\}$, we can use \eqref{est:faseRtau}, \eqref{recall:RtauDer} and \eqref{est:usandoFdBa} to find a constant $\tilde{C}>0$ independent of $\alpha, \beta$, $\gamma$ and $\tau$ such that
\begin{align}\label{interm}
  \displaystyle \big| \partial_x^\alpha\partial_t^\beta R^{\chi}_\tau[ u](x,t) \big| &\le \tilde{C} \bigg( \sum_{\gamma\leq \beta} h^{|\alpha|+|\gamma|} \tilde{h}^{|\beta-\gamma|}{\beta \choose \gamma} (|\alpha|+|\gamma|)!^{s}(|\beta- \gamma|)!^{s} \bigg) \,  \tau^{\tfrac m2}\, e^{s \tau^{1/s}-\tau R^2/33} \notag
\\
&\leq \tilde{C}    \tau^{\tfrac m2} e^{s \tau^{1/s}-\tau R^2/33} \big( h+\tilde{h}\big)^{|\alpha|+|\beta|} 2^{|\beta|} (|\alpha|+|\beta|)!^{s} , 
\end{align}
for every $(x,t) \in U.$
Thus, $R_{\tau}^{\chi}[u]$ converges to $0$  in $G^s(U)$ when $\tau$ converges to $\infty$.
\end{proof}

\subsection{Proof of Baouendi-Treves  approximation theorem  in $\mathcal{D}'_s$}\label{ap2} 

Given $\chi \in G_c^s(B_R^{\R^{m}}(0))$ we will first need to extend the definitions of $E_\tau^{\chi}[u], G_\tau^{\chi}[u]$ and $R_\tau^{\chi}[u]$ when $u \in {\mathcal{D}}'_s(W)$ is a solution of $\mathcal L$.

The definitions of  $E_\tau^{\chi}[u]$ and consequently $R_\tau^{\chi}[u]$ will strongly use the fact $u$ is a solution of $\mathcal L$ which guarantee that the pullback of $u$ to the submanifolds $\{(x,t): t=\text{constant}\}$ are well defined in the sense of ultradistributions, see Appendix~\ref{AppendixB}.

We can, however, provide a definition for $G_\tau^{\chi}[u]$ for every $u \in \mathcal{D}_s'(V)$ and as a consequence we will proof the convergence of $G_\tau^{\chi}[u]$ to $\chi u$ in $ {\mathcal{D}}_s'(V)$ when $\tau$ goes to $+\infty$ regardless wether $u$ is a solution of $\mathcal L$ or not.

\begin{Def}
Let $u \in {\mathcal{D}}_s'(V)$ and fix $\chi \in G^s_c(B_R^{\R^{m}}(0))$. We define $G_\tau^{\chi}[u]\in {\mathcal{D}}_s'(V)$, acting on $\varphi\in G^{s}_c(V)$ as
\begin{equation}\label{GtauD'}
  G_\tau^{\chi}[u](\varphi):= u_{(x't)}\big( \chi(x') G^{\tilde{\chi}}_\tau[\psi](x',t) \det Z_x(x',t)\big), 
\end{equation}
where $\tilde{\chi}$ is any element of  $G^{s}_c(B_R^{\R^{m}}(0))$ equal to $1$ in the projection of the support of $\varphi$ in $\R^{m}$ and $\psi(x,t):= \varphi(x,t)/\det Z_x(x,t).$  
\end{Def}

\begin{remark}\label{ConvGtauD'}
Note that, it follows immediately from \eqref{GtauD'} and Proposition~\ref{GtauFuncao} that for every $\varphi \in G^{s}(V)$, $G^{\chi}_\tau[u](\varphi)$ converges to $(\chi u)(\varphi)$, consequently, $G^{\chi}_\tau[u]$ converges to $\chi u$ in ${\mathcal{D}}_s'(V).$
\end{remark}

Now we will define $E_\tau^{\chi}[u]$ when $u \in {\mathcal{D}}_s'(V)$
is a solution of $\mathcal{L}$ in $W$. To do so, we will follow the results and notations from Appendix~\ref{AppendixB}, in particular, the definition of the trace of an ultradistribution \eqref{TRACE}.
Given $t\in B_{R}^{\R^{n}}(0)$ we can consider $\iota_t: B_{R}^{\R^m}(0) \lra V$ defined by $\iota_t(x)= (x,t)$. Since $u$ is a solution of $\mathcal{L}$ it holds that $WF_{s}(u) \cap \{(x,t, 0, \theta), (x,t) \in W, \tau \neq 0\}$ is empty. 
This means that, for each $t\in B_{R}^{\R^{n}}(0)$ we can define $\iota_t^{\ast}u\in {\mathcal{D}}_s'(B_R^{\R^{m}}(0))$, the trace of $u$ at $t$, by 
\begin{align}
  (\iota_t^{\ast}u)( \varphi )&= \frac{1}{(2 \pi)^{N}} \int_{\R^N} \bigg(\mathcal{F}(\lambda u)(\eta) e^{it \theta} \int_{\R^m} \varphi(x) e^{i x \sigma} \dd x\bigg) \dd \sigma \dd \theta, \quad \forall \varphi \in G^{s}_c(B_{R}^{\R^m}(0)),
\end{align}
where $\lambda\in G^{s}_c(W)$ is identically $1$ on $V$ and $\mathcal{F}(\lambda u)$ stands for the Fourier transform of $\lambda u.$ In the Appendix~\ref{AppendixB}, it is  shown that $\iota^{\ast}_tu$ is an ultradistribution with the property that, for each fixed $\varphi\in G^{s}_c(W)$, the application $B_R^{\R^n}(0) \ni t \mapsto \iota_t^{\ast}u(\varphi)$ is  of class $G^{s}.$

Moving on, notice that
\begin{align}
  \bigg(\frac{\tau}{\pi}\bigg)^{\tfrac m2}(\iota_0^{\ast}u)_{x'}\big( e^{-\tau\langle z- Z(x', 0)\rangle^{2}}  \chi(x') \det Z_x(x',0) \big)
\end{align}
is an entire function in $z\in \C^m$. So we can define $E_\tau^{\chi}[u]\in G^{s}(\R^N)$ as
\begin{align}\label{EtauD'}
  E_{\tau}^{\chi}[u](x,t):= \bigg(\frac{\tau}{\pi}\bigg)^{\tfrac m2}(\iota_0^{\ast}u)_{x'}\big( e^{-\tau\langle Z(x,t)- Z(x', 0)\rangle^{2}}  \chi(x') \det Z_x(x',0) \big).
\end{align}
Also, when $u$ is a solution of $\mathcal{L}$, one can verify that $G^{\chi}_\tau[u]$ given in \eqref{GtauD'} can be rewritten as
\begin{align}\label{OutraGtau}
  G^{\chi}_\tau[u](x,t)= \bigg(\frac{\tau}{\pi}\bigg)^{\tfrac m2}(\iota_t^{\ast}u)_{x'}\big( e^{-\tau\langle Z(x,t)- Z(x', t)\rangle^{2}}  \chi(x') \det Z_x(x',t) \big),
\end{align}
in this case one can check that $G^{\chi}_\tau[u] \in G^s(V)$, see Proposition~\ref{Leibiniz}.

Still assuming that $u$ is a solution of $\mathcal{L}$, we note that, for each $\varphi\in G^{s}_c(B_{R}^{\R^m}(0))$, we can use \eqref{OutraGtau} and then \eqref{eq Gtau} to write
\begin{align}\label{NewGtauD'}
  \int_{B_{R}^{\R^m}(0)} G^{\chi}_\tau[u](x,t) \varphi(x) \dd x&= \bigg(\frac{\tau}{\pi}\bigg)^{\tfrac m2}\!\!\int_{B_R^{\R^m}(0)}(\iota_t^{\ast}u)_{x'}\big( e^{-\tau\langle Z(x,t)- Z(x', t)\rangle^{2}}  \chi(x') \det Z_x(x',t) \big)\varphi(x) \dd x
\notag\\[4pt]
  &= (\iota_{t}^{\ast} u)_{x'} \big( G^{\tilde{\chi}}_{\tau} [\psi](x',t) \chi(x') \det Z_x(x',t)\big)
\end{align}
where $\psi(x,t)= \varphi(x)/\det Z_x(x,t)$ and $\tilde{\chi}\in G^{s}_c(B_R^{\R^{m}}(0))$ is equal to 1 on the support of $\varphi$. Thus, one can use Proposition~\ref{GtauFuncao} to obtain that for each fixed $t$, it holds
\begin{align}\label{GdoTraco}
 \lim_{\tau\lra \infty}  \int_{B_{R}^{\R^m}(0)} G^{\chi}_\tau[u](x,t) \varphi(x) \dd x &= (\iota_{t}^{\ast} u)_{x'} \big( \varphi(x')\chi(x') \big),  
\end{align}
for every $\varphi \in G^{s}_c(B_{R}^{\R^m}(0)).$
Moving on, we will now work on the error term $R^{\chi}_\tau[u]=E_{\tau}^{\chi}[u]-G_{\tau}^{\chi}[u]$ when $u\in{\mathcal D}'_s(W)$ is a solution of $\mathcal L$. The goal is to obtain an expression analogous to
\eqref{recall:Rtau}.

\begin{Pro}
Let $u\in{\mathcal D}_s'(W)$ be a solution of $\mathcal L$ and $\chi\in G^{s}_c(B^{\R^m}_{R}(0))$ then
\begin{align}\label{Fidentity}
  R^{\chi}_{\tau}[u](x,t)= \int_{[0,t]} \bigg(\frac{\tau}{\pi}\bigg)^{\tfrac m2} \sum_{j=1}^n   (\iota^{\ast}_{t'}u)_{x'}\big( e^{-\tau\langle Z(x,t)- Z(x', t')\rangle^{2}}\LL_j\chi(x') \det Z_x(x',t')\big)\dd t'_j 
\end{align}
for every $(x,t) \in V.$
\end{Pro}

\begin{proof}
To begin with, consider $\omega_{\tau,\tilde{\tau}}$ to be the sequence of $m-$forms with $G^{s}$ coefficients defined by
\begin{align*}
  \omega_{\tau,\tilde{\tau}}^{(x,t)}(x',t'):= \bigg(\frac{\tau}{\pi}\bigg)^{\tfrac m2}  e^{-\tau\langle Z(x,t)- Z(x', t')\rangle^{2}} G^{\chi}_{\tilde{\tau}}[u](x',t') \tilde{\chi}(x') \dd Z(x',t') 
\end{align*}
where $\dd Z= \dd Z_1\wedge \cdots\wedge \dd Z_m$ and $\tilde{\chi}\in G^{s}_c(B^{\R^m}_R(0))$ is equal to $1$ in $\supp \chi$. 
Also, let us define 
\begin{align*}
I_1^{\tau,\tilde{\tau}}(x,t)&=  \int_{\R^m\times [0,t]} \dd\omega_{\tau,\tilde{\tau}}^{(x,t)}(x',t'),\\
I_2^{\tau,\tilde{\tau}}(x,t)&=  \int_{\R^m} \omega_{\tau,\tilde{\tau}}^{(x,t)}(x',t),\ \text{and}\\
  I_3^{\tau,\tilde{\tau}}(x,t)&= \int_{\R^m} \omega_{\tau,\tilde{\tau}}^{(x,t)}(x',0).
\end{align*}
Thanks to Stokes' theorem it holds that $I_1^{\tau,\tilde{\tau}}(x,t)= I_2^{\tau,\tilde{\tau}}(x,t)- I_3^{\tau,\tilde{\tau}}(x,t).$
Now for any given $\varphi_1\in G^s_c(B^{\R^m}_R(0))$, $\varphi_2\in G^{s}_c(B^{\R^n}_R(0))$. Applying $I_2^{\tau,\tilde{\tau}}$ to $\varphi_1 \otimes \varphi_2$ in the sense of ultradistributions we obtain
\begin{align*}
   I_2^{\tau,\tilde{\tau}} (\varphi_1 \otimes \varphi_2)= \int_{B^{\R^m}(0)} \mathcal{I}^{\tau,\tilde{\tau}}_{\varphi_2}(x)\varphi_1(x) \dd x, 
\end{align*}
where
\begin{align*}
  \mathcal{I}^{\tau,\tilde{\tau}}_{\varphi_2}(x):= \bigg(\frac{\tau}{\pi}\bigg)^{\tfrac m2} \int_{B_R^{\R^n}(0)}\int_{\R^m} e^{-\tau \langle Z(x,t)- Z(x', t)\rangle^{2}} G^{\chi}_{\tilde{\tau}}[u](x', t) \tilde{\chi}(x') \varphi_2(t)\det Z_x(x',t) \dd x' \dd t.
\end{align*}
Thus, one can use \eqref{GdoTraco} to conclude that
\begin{align}\label{===}
  \lim_{\tilde{\tau} \lra \infty} \mathcal I^{\tau, \tilde{\tau}}_{\varphi_2}(x)&= \bigg(\frac{\tau}{\pi}\bigg)^{\tfrac m2} \int_{B_R^{\R^{n}}(0)}( \iota_t^{\ast} u)_{x'} \bigg(e^{-\tau \langle Z(x,t)- Z(x', t)\rangle^{2}}  \chi(x') \varphi_2(t)\det Z_x(x',t) \bigg)\dd t \notag
  \\  &=  \int_{B^{\R^n}_R(0)} G_\tau^{\chi}[u](x,t)  \varphi_2(t)\dd t.
\end{align}
If follows from identity \eqref{===} that $I^{\tau, \tilde{\tau}}_2$ converges to $G^{\chi}_\tau[u]$ in $\mathcal{D}_s'(V)$ as $\tilde{\tau} \lra \infty$. Analogously, $I_3^{\tau, \tilde{\tau}}$ converges to $E^{\chi}_\tau[u]$ in $\mathcal{D}_s'(V)$ as $\tilde{\tau} \lra \infty$. 
Therefore, all we have to do now is to focus on the next identity
\begin{align*}
  R^{\chi}_\tau[u](\varphi)= \lim_{\tilde{\tau} \lra \infty} I_{1}^{\tau,\tilde{\tau}}(\varphi), 
\end{align*}
for every $ \varphi \in G^{s}_c(V).$
Note that $ \dd\omega_{\tau, \tilde{\tau}}^{(x,t)}$ can be written as
\begin{align*}
   \dd \omega_{\tau,\tilde{\tau}}^{(x,t)}(x',t')= \bigg(\frac{\tau}{\pi}\bigg)^{\tfrac m2} \sum_{j=1}^n e^{-\tau\langle Z(x,t)- Z(x', t')\rangle^{2}} \LL_j\big( G^{\chi}_{\tilde{\tau}}[u](x',t') \tilde{\chi}(x')\big) \dd t'_j\wedge \dd Z(x',t') 
\end{align*}
and using the following equality
\begin{align*}
\LL_j \big(G^{\chi}_{\tilde{\tau}}[u](x',t')\tilde{\chi}(x')\big)= G^{\LL_j\chi}_{\tilde{\tau}}[u](x',t')\tilde{\chi}(x')+ G^{\chi}_{\tilde{\tau}}[\LL_j u](x',t')\tilde{\chi}(x')+ G^{\chi}_{\tilde{\tau}}[u](x',t')\LL_j\tilde{\chi}(x') 
\end{align*}
together with the convergence stated in \eqref{GdoTraco} we obtain
\begin{align*}
  R^{\chi}_\tau[u](\varphi) = \int_V \int_{[0,t]} \bigg(\frac{\tau}{\pi}\bigg)^{\tfrac m2} \sum_{j=1}^n   (\iota^{\ast}_{t'}u)_{x'}\big( e^{-\tau\langle Z(x,t)- Z(x', t')\rangle^{2}}\LL_j\chi(x') \det Z_x(x',t')\big) \varphi(x,t)\dd t'_j \dd x \wedge \dd t
\end{align*}
since $u$ is a solution of $\mathcal{L}$ and $\LL_j \tilde{\chi}=0 $ over $\supp \chi$,
as we wished to prove.
\end{proof}

Now we can use  \eqref{Fidentity} to conclude the proof of Theorem~~\ref{BTAPPthm},  in $\mathcal{D}_s'$.

\begin{Pro}\label{RtauD'}
 Let $\chi\in G^{s}_c(B^{\R^m}_{R}(0))$, with $\chi=1$ in $B^{\R^m}_{R/2}(0)$ and $u \in \mathcal{D}_s'(W)$ a solution of $\mathcal{L}$. Then $R^{\chi}_\tau[u]$ converges to $0$ in $G^s(U)$ when $\tau \lra \infty$. 
\end{Pro}

\begin{proof}
For every $j= 1, \ldots, n$, we define
\begin{align*}
  \varPhi_j(x,t,x',r)= e^{-\tau\langle Z(x,t)- Z(x',  rt)\rangle^{2}}\LL_j\chi(x')\det Z_x(x',r t) t_j.
\end{align*}
Note that there exist $\rho>0$ such that $\varPhi_j\in G^{s,\rho}(U\times B_{R}^{\R^{m}}(0)\times (0,1))$ for every $j=1, \ldots, n$.
Now we differentiate $R^\chi_\tau[ u](x,t)$ using \eqref{Leibiniz} to obtain 
\begin{align}\label{eq;der-traco}
  \del_{x}^{\alpha} \del_t^{\beta}R^{\chi}_{\tau}[u](x,t) = \bigg(\frac{\tau}{\pi}\bigg)^{\tfrac m2} \sum_{j=1}^n  \int_{0}^{1} \sum_{\gamma \leq \beta} {\beta \choose \gamma}  \del_t^{\beta-\gamma} (\iota^{\ast}_{rt}u)_{x'}\big(  \del_x^{\alpha} \del_t^{\gamma} \big\{\varPhi_j(x, t, x',r)\big\}\big) \dd r.  
\end{align}
 Thus we can consider $\del_x^{\alpha} \del_t^{\gamma} \varPhi_j$  as an element of $G^{s, \tilde{\rho}}(B_{R}^{\R^m}(0))$ in $x'\in B_{R}^{\R^m}(0)$ (where $\tilde{\rho}$ could be any number greater than $\rho$, let us take $\tilde{\rho}=2^s\rho$) and we can apply  estimate~\eqref{the-end} from Appendix~\ref{AppA}  to obtain
\begin{align}\label{Traceinequality1}
  \del_t^{\beta-\gamma} (\iota^{\ast}_{st}u)_{x'}\big(  \del_x^{\alpha} \del_{t}^{\gamma}\big\{\varPhi_j(x, t, x',r)\big\}\big)
  &\leq \|\del_x^{\alpha} \del_{t}^{\gamma} \varPhi_j\|_{\tilde{\rho}, B_{R}^{\R^{m}}(0)}C  \tilde{H}^{|\beta-\gamma|} |\beta-\gamma|!^{s}.
\end{align}
Now for every $(x, t) \in U$ and each $r\in (0,1)$ we have 
\begin{align}\label{Traceinequality1'}
  \|\del_x^{\alpha} \del_{t}^{\gamma}\varPhi_j\|_{ \tilde{\rho}, B_{R}^{\R^{m}}(0)}&= \sup_{\theta\in \Z_+^m}  \sup_{x' \in B_{R}^{\R^{m}}(0)} \frac{ \left| \del_{x'}^{\theta}\del_x^{\alpha} \del_{t}^{\gamma} \big\{\varPhi_j(x, t, x',r)\big\}\right|}{ \tilde{\rho}^{|\theta|}|\theta|!^{s}}
  								\notag\\
                                                                     &= \sup_{\theta\in \Z_+^m}  \sup_{x' \in B_{R}^{\R^{m}}(0)} \bigg\{\frac{ \left|\del_{x'}^{\theta}\del_x^{\alpha} \del_{t}^{\gamma} \big\{\varPhi_j(x, t, x',r)\big\}\right|}{ \rho^{|\theta|+|\alpha|+ |\gamma|} (|\theta|+ |\alpha|+|\gamma|)!^{s}} \frac{(|\theta|+ |\alpha|+|\gamma|)!^{s} \rho^{|\alpha|+ |\gamma|}}{2^{s|\theta|}|\theta|!^s}\bigg\}
                                                                     \notag\\
                                                                     &\leq \sup_{\theta\in \Z_+^m}  \sup_{x' \in B_{R}^{\R^{m}}(0)} \bigg\{\frac{\left|\del_{x'}^{\theta}\del_x^{\alpha} \del_{t}^{\gamma} \big\{\varPhi_j(x, t, x',r)\big\}\right|}{ \rho^{|\theta|+|\alpha|+ |\gamma|} (|\theta|+ |\alpha|+|\gamma|)!^{s}}  \bigg\} (|\alpha|+ |\gamma|)!^{s} (2^s\rho)^{|\alpha|+ |\gamma|}
                                                                     \notag\\
  &\leq \|\varPhi_j\|_{\rho,U\times B_{R}^{\R^{m}}(0)\times (0,1)} (|\alpha|+ |\gamma|)!^{s} (2^s\rho)^{|\alpha|+ |\gamma|}.
\end{align}
Let us denote by $\mathcal{U}=U\times B_{R}^{\R^{m}}(0)\times (0,1)$, then using \eqref{eq;der-traco}, \eqref{Traceinequality1} and \eqref{Traceinequality1'} we obtain
\begin{align}\label{Est.RtauDerivative}
  |\del_{x}^{\alpha} \del_t^{\beta}R^{\chi}_{\tau}[u](x,t)|&\leq  \bigg(\frac{\tau}{\pi}\bigg)^{\tfrac m2} \sum_{j=1}^n \sum_{\gamma \leq \beta} {\beta \choose \gamma}  \|\varPhi_j\|_{\rho,\mathcal{U}} (|\alpha|+|\gamma|)!^s (2\rho)^{|\alpha+ \gamma|}C \tilde{H}^{|\beta-\gamma|}|\beta-\gamma|!^{s}
\notag\\[4pt]
  &\leq   C \bigg(\frac{\tau}{\pi}\bigg)^{\tfrac m2}  \Big(\sum_{j=1}^{n} \|\varPhi_j\|_{\rho,\mathcal{U}}\Big) (2^s(2^s\rho+ \tilde{H}))^{|\alpha+ \beta|}   (|\alpha|+|\beta|)!^{s}.
\end{align}
Now let us estimate $\|\varPhi_j\|_{\rho,U\times B_{R}^{\R^{m}}(0)\times (0,1)}$, we have
\begin{align}\label{eq:WritingDerivatives}
  &\del_{x}^{\alpha}\del_t^{\beta} \del_{x'}^{\gamma} \del_{s}^{\sigma}\varPhi_j(x,t,x',r)=
\\[4pt]
  &=\sum_{S_{\beta, \gamma, \sigma}} {\beta\choose \beta'} { \gamma \choose \gamma'} {\sigma \choose \sigma'} \del_{x}^{\alpha}\del_t^{\beta'} \del_{x'}^{\gamma'} \del_{r}^{\sigma'} e^{-\tau\langle Z(x,t)- Z(x',  rt)\rangle^{2}} \del_t^{\beta''} \del_{x'}^{\gamma''} \del_{r}^{\sigma''}\Lambda(t,x', r),
\notag
\end{align}
where $S_{\beta, \gamma, \sigma}= \{(\beta', \beta'', \gamma', \gamma'', \sigma', \sigma''): \beta'+ \beta''= \beta; \gamma'+ \gamma''= \gamma; \sigma'+ \sigma''=\sigma\}$ and
\begin{align*}
  \Lambda(t,x', r)= \big(\LL_j\chi(x')\det Z_x(x',r t) t_j\big).
\end{align*}
Using Lemma~\ref{lem:usandoFdB}
with  $f(x,t,x',r)= \langle Z(x,t)-Z(x',rt)\rangle^2$ we see that there are constants $C'>0$ and $h>0$ such that
\begin{align}\label{expinequalityR}
\frac{\left| \del_{x}^{\alpha}\del_t^{\beta'} \del_{x'}^{\gamma'} \del_{r}^{\sigma'} e^{-\tau\langle Z(x,t)- Z(x',  rt)\rangle^{2}}\right|}{ h^{|\alpha|+|\beta'|+|\gamma'|+|\sigma'|} (|\alpha|+|\beta'|+|\gamma'|+|\sigma'|)!^{s} }
&\leq C' \,e^{-\tau\Re \langle Z(x,t)-Z(y,rt) \rangle^2+ s \tau^{1/s}}.
\end{align}
Also, there is a constant $\tilde{h}>0$ such that $\Lambda \in G^{s, \tilde{h}}(B_{R}^{\R^{n}}(0)\times B_{R}^{\R^{m}}(0)\times (0,1))$ and so there is $\tilde{C}>0$ such that
\begin{align}\label{inequalityLambdaR}
  |\del_t^{\beta''} \del_{x'}^{\gamma''} \del_{r}^{\sigma''}\Lambda(t,x', r)|\leq \tilde{C} \tilde{h}^{|\beta''|+|\gamma''|+ |\sigma''|}(|\beta''|+|\gamma''|+ |\sigma''|)!^{s}.
\end{align}
Consequently, using \eqref{eq:WritingDerivatives}, \eqref{expinequalityR}, \eqref{inequalityLambdaR} and \eqref{est:faseRtau}, we obtain
\begin{align}\label{est.Der}
  \frac{ |\del_{x}^{\alpha}\del_t^{\beta} \del_{x'}^{\gamma} \del_{r}^{\sigma}\varPhi_j(x,t,x',r)|}{\big[2\big( h+ \tilde{h}\big)\big]^{|\alpha|+|\beta|+|\gamma|+|\sigma|}  (|\alpha|+|\beta|+|\gamma|+|\sigma|)!^{s} }  &\leq  C'\tilde{C} e^{-\tau\Re \langle Z(x,t)-Z(y,rt) \rangle^2+ s \tau^{1/s}}\notag \\
 & \leq  C'\tilde{C} e^{-\tau R/33+ s \tau^{1/s}}.
\end{align}
Therefore we can take $\rho= 2(h+\tilde{h})$, it follows from \eqref{Est.RtauDerivative} and \eqref{est.Der} that
\begin{align*}
 \frac{ |\del_{x}^{\alpha} \del_t^{\beta}R^{\chi}_{\tau}[u](x,t)|}{(2(2\rho+ \tilde{H}))^{|\alpha|+ |\beta|}   (|\alpha|+|\beta|)!^{s}}  &\leq  \frac{C C'\tilde{C}}{\pi^{\frac{m}{2}}} \tau^{\frac{m}{2}} e^{-\tau R/33+ s \tau^{1/s}}.
\end{align*}
Proving that $R_{\tau}^{\chi}[u]$ converges to $0$ in $G^{s}(U)$ as desired.
\end{proof}

\section{Approximate Poincaré Lemma \label{AAPL}}

 Let $\Omega$ be an open neighborhood of the origin in $\RN$ and assume that we have a locally integrable structure in $\Omega$ where the orthogonal $\mathcal{L}^{\perp}$ is defined globally in $\Omega$ by the differential of $Z_1,\ldots, Z_m$ and denote $\lambda(x, t)= (Z_1(x,t), \ldots, Z_m(x,t),t_1, \ldots, t_n)$.
We define $G^{s}(\Omega, \Lambda^{p,q})$ the space of all $(p,q)$-forms
\begin{align}\label{Expressaopraf}
  f(x,t)= \sum_{|I|=p}\sum_{|J|=q} f_{IJ}(x,t) \dd Z_I\wedge \dd t_J,
\end{align}
where the coefficients $f_{IJ}\in G^{s}(\Omega)$, where $\dd Z_I = \dd Z_{i_1} \wedge \ldots \wedge \dd Z_{i_p}$ and $\dd t_J= \dd t_{j_1}  \wedge \ldots \wedge \dd t_{j_q}$ for $I=\{1\leq i_1< \cdots< i_p\leq m\}$ and $J=\{1\leq j_1< \cdots< j_q\leq n\}$.  The notation $K=\{ 1\leq k_1< \cdots< k_s\leq r\}$ means that $K= \{k_1, \ldots, k_s\}\subset \{1, \ldots, r\}$ and that $k_1< \cdots< k_s$.

Let us define a linear differential operator  $\mathbb{L}: G^{s}(\Omega, \Lambda^{p,q}) \lra G^{s}(\Omega, \Lambda^{p,q+1})$ by
\begin{align*}
  \mathbb{L}f   &= \sum_{|I|=p}\sum_{|J|=q} \sum_{j=1}^{n} \LL_j f_{IJ} \dd t_j\wedge\dd Z_I\wedge\dd t_J, 
\end{align*}
for every $ f\in G^{s}(\Omega, \Lambda^{p,q}).$ The Gevrey local solvability of $\mathbb{L}$ in degree $(p,q)$ at a point $p_0\in \Omega$ here means that there is $\Omega_0$ a neighborhood of $p_0$ such that for any other neighborhood $\Omega_1$ of $p_0$ with $\Omega_1\subset \Omega_0$, we can find a neighborhood $\Omega_2$ of $p_0$ such that $\Omega_2\subset \Omega_1$ and for any $f\in G^{s}(\Omega_1,\Lambda^{p,q})$ such that $\mathbb{L}f=0$ there is $g \in G^{s}(\Omega_2, \Lambda^{p,q-1})$ such that $\mathbb{L}g= f$ in $\Omega_2$.

We recall that the approximate Poincaré lemma is a result concerning approximate solvability of $\mathbb{L}$. Before we actually enunciate and prove this result let us fix more notation and recall an important trick. 
Let $J=\{1\leq j_1< \cdots< j_q\leq n\}$ and $j\in \{1, \ldots, n\}\setminus J$ and we define $\epsilon(j, J)$ to be the sign of permutations to ordenate the $q+1$-form $\dd t_j \wedge \dd t_J$, i.e, $\epsilon(j, J)$ is $1$ if the number of permutation is even and $-1$ if this number is odd.  

Assume that $q\geq 2$ and define, for any $J$ with $|J|=q$, the $q-1$-form
\begin{align*}
  \omega_J= \sum_{j\in J} \epsilon(j, J\setminus \{j\}) t_j \dd t_{J\setminus\{{j}\}},
\end{align*}
and, when $q=1$, $\omega_J= t_J$.

Now we follow \cite{treves_has}, and, for any $q$-form
\begin{align*}
  F= \sum_{|J|=q} F_J \dd t_J,
\end{align*}
we define an operator for $q$-forms to $q-1$-forms
\begin{align*}
  K^{(q)}F= \sum_{|J|=q} \Big\{\int_0^{1} F_J(\sigma t)\sigma^{q-1}\dd \sigma\Big\} \omega_J.
\end{align*}
This operator satisfies the following formula:
\begin{align}\label{formulamagica}
  F= \dd_t K^{q}F+ K^{(q+1)} \dd_t F.
\end{align}

Assume that $W$ and $V$ are as in the Baouendi-Treves approximation formula, i.e., $V= B_{R}^{\Rm}(0)\times B_{R}^{\Rn}(0)$, $V\subset \subset W\subset \subset\Omega$ and such that \eqref{eq z4} holds and let $U= B_{R/2}^{\Rm}(0)\times B_{R}^{\Rn}(0)$.  For every  $\chi\in G^{s}_c(B_R^{\R^{m}}(0))$ and $g\in G^s(\Omega)$ we define
\begin{equation}\label{NovaeqpraG}
  \mathcal{G}_\tau^{\chi}[g](z,t): = \Big(\frac\tau{\pi}\Big)^{\tfrac m2}\!\!  \int  e^{-\tau \langle z-Z(x',t)\rangle^2} \chi(x')g(x', t) \, {\rm det} Z_x(x',t) \dd x'.
\end{equation}
 Note that $\mathcal{G}_{\tau}^{\chi}[g](Z(x,t), t)= G^{\chi}_\tau[g](x,t)$. If $f$ is a $(p,q)$-form as in \eqref{Expressaopraf} we define
\begin{align}\label{ExpressaopraGf}
  \mathcal{G}^{\chi}_\tau [f](z,t)= \sum_{|I|=p}\sum_{|J|=q}  \mathcal{G}_\tau^{\chi}[f_{IJ}](z,t)\dd z_I\wedge \dd t_J,
\end{align}
then $\lambda^{\ast}(\mathcal{G}^{\chi}_\tau [f])(x,t)= G^{\chi}_\tau [f](x,t)$ where $G^{\chi}_\tau [f]$ is defined by allowing $G^{\chi}_\tau$ acts coefficientwise. 

We now can define
\begin{align*}
  \mathcal{K}_\tau^{(p,q)} [f,\chi](z,t)&= (-1)^{p} K^{(q)} \mathcal{G}^{\chi}_\tau[f](z,t)\\
  &= (-1)^{p}\sum_{|I|=p}  \sum_{|J|=q} \Big\{ \int_{0}^{1}\mathcal{G}_\tau^{\chi}[f_{IJ}](z,\sigma t)\sigma^{q-1} \dd \sigma\Big\} \dd z_I\wedge \omega_J
\end{align*}
It follows from \eqref{formulamagica} that we have
\begin{align}\label{segundaformulamagica}
  \mathcal{G}^{\chi}_\tau[f]=  (-1)^{p} \dd_{t}[ K^{(q)} \mathcal{G}_\tau^{\chi} [f]]+ (-1)^{p}   K^{(q+1)} \dd_t \mathcal{G}^{\chi}_\tau[f].
\end{align}

\begin{Theo}
  Assume $0\leq p\leq m, 1\leq q\leq n.$ There are open neighborhoods of the origin, $W$ and $U$ as above, such that given any $f\in G^{s}(\overline{W}; \Lambda^{p,q})$ that is $\mathbb{L}$-closed and any $\chi\in G^{s}_c(B_{R}^{\R^{m}}(0))$ that is equal to $1$ in $B_{R/2}^{\R^{m}}(0)$ we have
  \begin{align*}
    f= \lim_{\tau \lra \infty} \mathbb{L} \big(\lambda^{\ast}\mathcal{K}_\tau^{(p,q)} [f,\chi]\big)
  \end{align*}
  in $G^{s}(U,\Lambda^{p,q})$.
\end{Theo}
\begin{proof}
  From \eqref{segundaformulamagica}, it follows that it is enough to prove that $ K^{(q+1)} \dd_t \mathcal{G}^{\chi}_\tau[f]$ converges to $0$ in $G^{s}(V,\Lambda^{p,q})$, note that $\mathbb{L} \lambda^{\ast}= \lambda^{\ast} \dd_t$. Now we use that
  \begin{align*}
  \dd_t \mathcal{G}_\tau^{\chi}[f](z,t)= \sum_{|I|=p}\sum_{|J|=q} \sum_{j=1}^{n}\Big( \mathcal{G}_\tau^{\LL_j\chi}[f_{IJ}](z,t)+\mathcal{G}_\tau^{\chi}[\LL_jf_{IJ}](z,t)\Big) \dd t_j\wedge \dd z_I\wedge \dd t_J.    
  \end{align*}
  Since  $f$ is $\mathbb{L}$-closed it follows, for every $|I|=p$ and every $|K|=q+1$, that 
  \begin{align*}
  \Big(\frac\tau{\pi}\Big)^{\tfrac m2}\!\!  \int  e^{-\tau \langle z-Z(x',t)\rangle^2} \chi(x')\sum_{\substack{K=J\cup\{j\}\\|J|=q}}  \epsilon(j,J) \LL_j f_{IJ}(x', t) \, {\rm det} Z_x(x',t) \dd x'= 0.
  \end{align*}
  Consequently,
  \begin{align*}
  \sum_{|I|=p}\sum_{|J|=q} \sum_{j=1}^{n}\mathcal{G}_\tau^{\chi}[\LL_jf_{IJ}](z,t) \dd t_j\wedge \dd z_I\wedge \dd t_J=0.
  \end{align*}

  Therefore all we need to show is that
  \begin{align*}
    \sum_{|I|=p}\sum_{|J|=q} \sum_{j=1}^{n}\lambda^{\ast}K^{(q+1)} \Big(\mathcal{G}_\tau^{\LL_j\chi}[f_{IJ}] \dd t_j\wedge \dd z_I\wedge \dd t_J\Big)(x,t) \lra 0 \textrm{ in } G^s(U; \Lambda^{p,q}).
  \end{align*}
Since
  \begin{align}\label{ultimaequacao}
   \lambda^{\ast}K^{(q+1)} \Big(\mathcal{G}_\tau^{\LL_j\chi}[f_{IJ}] \dd t_j\wedge \dd z_I\wedge \dd t_J\Big)(x,t)= (-1)^{p}\Big( \int_{0}^{1}\mathcal{G}_\tau^{\LL_j\chi}[f_{IJ}](Z(x,t),\sigma t)\sigma^{q} \dd \sigma\Big) \dd Z_I\wedge \omega_J
  \end{align}
one can  we use  the same argument to prove that $R^{\chi}_\tau[ u](x,t)$ given by \eqref{recall:Rtau} converges to $0$ in $G^{s}(U)$ to conclude that coefficient of the form in the left hand side of \eqref{ultimaequacao} converges to $0$ in $G^{s}(U)$. 
\end{proof}

\section{Ultradistributions vanishing on maximally real submanifolds \label{AUVOMRS}}

Let $\Omega$ an open subset of $\RN$ and  $\mathcal{L}$ be a locally integrable structure of corank $m$. Let $\Sigma\subset \Omega$ be an embedded Gevrey submanifold of dimension $m$, i.e., the defining functions of  $\Sigma$  are Gevrey functions. We recall that $\Sigma$ is  maximally real with respect to $\mathcal{L}$ if for every $p\in \Sigma$, any nonvanishing section of $\mathcal{L}$ defined in a neighborhood of $p$ is transversal to $\Sigma$ at $p.$

\begin{Theo}
  Let $\Sigma$ be an embedded submanifold in $\Omega$  maximally real with respect to $\mathcal{L}$. If $u\in \mathcal{D}'_s(\Omega)$ is a solution of $\mathcal{L}$ and $u|_{\Sigma}=0$, then $u$ vanishes in a neighborhood of $\Sigma.$
\end{Theo}
\begin{proof}
  It is enough to prove that every  $p\in \Sigma$ has a neighborhood where $u$ vanishes. Fix $p\in \Sigma$ so that we can find local coordinates $(x,t)$ centered at $p$ and $Z_1, \ldots, Z_m$ such that properties \eqref{eq z1} and \eqref{eq z2} hold and $\Sigma=\{(x, 0)\}$ in  a neighborhood of $p$ as proved in \cite{EaGr03}. 
 
Now thanks to the Baouendi-Treves approximation formula there is $U$ a neighborhood of $p$ where $u$ is the limit of 
  \begin{align*}
  E_{\tau}^{\chi}[u](x,t)= \bigg(\frac{\tau}{\pi}\bigg)^{\tfrac m2}(\iota_0^{\ast}u)_{x'}\big( e^{-\tau\langle Z(x,t)- Z(x', 0)\rangle^{2}}  \chi(x') \det Z_x(x',0) \big).
\end{align*}

Since $\Sigma= \{(x,0)\}$, $u|_{\Sigma}=0$ means that $\iota_{0}^{\ast} u=u|_{\Sigma}=0$. Therefore, $  E_{\tau}^{\chi}[u](x,t)$ vanishes in a neighborhood of $p$ and so does $u$.
\end{proof}

\section{Baouendi-Treves theorem in Denjoy-Carleman classes}\label{BTDC}
One can use the ideas of Section~\ref{ap} to  proof of the Baouendi-Treves approximation theorem for more general classes of ultradifferentiable functions. 
More precisely, consider 
the strongly non-quasianalytic Denjoy-Carleman classes of  Roumieu type associated with a non-decreasing sequence of positive numbers $(M_p)_{p\in \Z_+}$ satisfying:
\begin{itemize}  
\item  \underline{\textit{Initial condition:}},
 \begin{equation} \label{P1} 
M_0=M_1=1.
\end{equation}
\item \underline{\textit{Strong logarithmic convexity:}} 
\begin{equation} \label{logconvex}
\frac{M_j}{M_{j-1}} \leq  \frac{M_{j+1}}{M_j} , \qquad j=1,2,3,\ldots.
\end{equation}
\item \underline{\it{Stability under ultradifferential operators:}} There exist $A, \,H>0$ such that 
\begin{equation}\label{M2'} 
M_{j+k}\leq A H^{j+k} M_j M_k, \qquad\forall\, j,k \in \Z_+.
\end{equation} 
\item \underline{\it{strong Non-quasianalyticity condition:}} there exist a constant $A>0$ such that 
\begin{equation}\label{2}
\sum_{j=p+1}^\infty \frac{M_{j-1}}{M_j}<A p\frac{M_p}{M_{p+1}}, \qquad p=1,2,3,\ldots
\end{equation}
\end{itemize}
We refer to \cite{Ko} for more details about these classes. 
The techniques used strongly the fact that $G^s(\overline{V})$ and $G^s(\overline{V}; \MM, \LL)$ are isomorphic as topological spaces. This result can be adapted to strongly non-quasianalytic Denjoy-Carleman classes.
With this equality of topological spaces proved it is not difficult to see that with minor changes in our proof the Baouendi-Treves approximation theorem also holds for these spaces of strongly non-quasianalytic Denjoy-Carleman functions and ultradistributions.

\appendix 

\section{Faà di Bruno formula}\label{AppA}

Next we recall the Faà di Bruno generalized formula.

\begin{Theo}[\cite{bierstonemilman}]\label{denerlzdibruno} 
Let $\Omega\subset \R^{p}$ and $U\subset \R^{n}$ open subsets. Let $f\in \Cinf(\Omega)$ and $g\in \Cinf(U; \R^{p})$ such that $g(U)\subset \Omega$ and denote by $h$ the composition $f\circ g$.
For all $\alpha \in \Z_+^{n}\setminus\{0\}$, we have that
\begin{align*}
  \del^{\alpha} h(x)= \sum_{\mathcal{S_{\alpha}}}   \del^{\kappa} f(g(x))\alpha! \frac{ \big(\del^{\delta_1} g(x)\big)^{\beta_1}}{\beta_1! \delta_1!^{|\beta_1|}} \cdots \frac{ \big(\del^{\delta_\ell} g(x)\big)^{\beta_\ell}}{\beta_\ell! \delta_\ell!^{|\beta_\ell|}},
\end{align*}
where  $\kappa= \beta_1 + \cdots +\beta_\ell$ and $\mathcal{S}_\alpha$ is the set of all $\{\delta_1, \ldots, \delta_\ell\}$ distinct elements of $\big( \Z_+^{n}\setminus\{0\}\big)^{\ell}$ and all $(\beta_1, \ldots, \beta_\ell)\in \big(\Z^{p}_+\setminus\{0\}\big)^{\ell}$, $\ell =1, 2, 3, \ldots, $ such that
\begin{align*}
  \alpha= \sum_{j=1}^{\ell} |\beta_j|\delta_j.
\end{align*}
\end{Theo}

 We will also need the following result from \cite{bierstonemilman}: 
\begin{Lem}\label{lem:fromBM}
Given $\alpha\in \Z_+^{n}\setminus\{0\}$ and $p \in \Z_+\setminus\{0\}$ let $\mathcal{S}_{\alpha}$ the set defined in the Theorem~\ref{denerlzdibruno}. For every $(\beta_1, \dots , \beta_{\ell};\delta_1, \dots , \delta_\ell) \in \mathcal{S}_\alpha$, we have
\begin{equation}\label{sk=r}
|\kappa|!^{t}\, |\delta_1|!^{t \beta_1} \ldots |\delta_\ell|!^{t\beta_\ell} \leq |\alpha|!^{t}
\end{equation}
for every $t>0$ and, for every positive constant $A$, there are constants $L,D>0$, depending only on $A$, $n$ and $p$, such that
\begin{equation}
  \sum_{\mathcal{S}_\alpha} \frac{\kappa!}{\beta_1! \dots \beta_\ell! } A^{|\kappa|} \le LD^{|\alpha|}.
\end{equation}
\end{Lem}

\begin{Lem} \label{lem:usandoFdB} 
Let  $\Omega \subset \R^n \times \R^m$ be an open  set, $f\in G^{s}(\Omega)$ and $\tau>1$ is a parameter. Then, for each compact subset $K  \subset \Omega$   there exist constants $C,h>0$ 
such that and $\alpha \in \Z_+^m$, it holds
\begin{equation}\label{est:usandoFdB}
\sup_{(x, y)\in K}\left|\partial_x^{\alpha} \! \left\{ e^{\tau f(x,y)} \right\} \right|  
         \leq C h^{|\alpha|} |\alpha|!^s  \,e^{\tau \Re f(x,y)+ s \tau^{1/s}}.
\end{equation}
\end{Lem}
\begin{proof}
It is enough to prove for $\theta\in (0,1]$. Using Theorem \ref{denerlzdibruno}, we have
\begin{align*}
  \del^{\alpha} \left\{ e^{\tau f(x,y)} \right\}= \sum_{\mathcal{S_{\alpha}}}   e^{\tau f(x,y)}\alpha! \frac{\tau \big(\del^{\delta_1} f(x,y)\big)^{\beta_1}}{\beta_1! \delta_1!^{\beta_1}} \cdots \frac{\tau \big(\del^{\delta_\ell} f(x,y)\big)^{\beta_\ell}}{\beta_\ell! \delta_\ell!^{\beta_\ell}}.
\end{align*}
Since  $f$ is in  $G^s(\Omega)$ 
there exist constants $\tilde{C},\tilde{h}>0$ such that,  
\begin{align*}
\left|\partial_x^{\alpha} \! \left\{ e^{\tau f(x,y)} \right\} \right|  &\leq \alpha! \,e^{\tau \Re f(x,y)} \! \sum_{\mathcal{S}_\alpha}
\frac{\tau^{\kappa}}{\beta_1! \dots \beta_\ell! } 
\prod_{j=1}^{\ell} \frac{ \big| \partial^{\delta_j}_x\left\{f(x,y)\right\}\big|^{\beta_j}}{\delta_j!^{\beta_j}}  \\[5pt]
  &\leq \alpha! \,e^{\tau \Re f(x,y)} \! \sum_{\mathcal{S}_\alpha}
\frac{\tau^{\kappa}}{\beta_1! \dots \beta_\ell! } 
\prod_{j=1}^{\ell} \frac{ \big( \tilde{C}\tilde{h}^{|\delta_j|} \delta_j!^s\big)^{\beta_j}}{\delta_j!^{\beta_j}} 
     \\[5pt]
  &= \tilde{h}^{|\alpha|} \alpha! \,e^{\tau \Re f(x,y)} \! \sum_{\mathcal{S}_\alpha}
\frac{(\tilde{C}\tau)^{\kappa}}{\beta_1! \dots \beta_\ell! } 
\prod_{j=1}^{\ell}   \delta_j!^{(s-1)|\beta_j|} 
     \\[5pt]
         &\leq  \tilde{h}^{|\alpha|} \alpha! \,e^{\tau \Re f(x,y)} \!  \sum_{\mathcal{S}_\alpha}
\frac{(\tilde{C} \tau)^{\kappa}}{\beta_1! \dots \beta_\ell! } \bigg(\frac{|\alpha|!}{ \kappa!}\bigg)^{s-1} \\
         &\leq  \tilde{h}^{|\alpha|} |\alpha|!^{s} \,e^{\tau \Re f(x,y)} \!  \sum_{\mathcal{S}_\alpha} 
\frac{\tilde{C}^{|\kappa|}\kappa!}{\beta_1! \dots \beta_\ell! } \bigg(\frac{\tau^{\kappa/s}}{ \kappa!}\bigg)^{s} \\
&\leq  \tilde{h}^{|\alpha|} |\alpha|!^{s} \,e^{\tau \Re f(x,y)+s \tau^{1/s}} \!  \sum_{\mathcal{S}_\alpha} 
\frac{\kappa!}{\beta_1! \dots \beta_\ell! }\tilde{C}^{|\kappa|}.
\end{align*}

Now we use Lemma~\ref{lem:fromBM} to find constants $C$ and $h$ such that
\begin{align*}
\left|\partial_x^{\alpha} \! \left\{ e^{\tau f(x,y)} \right\} \right| \leq Ch^{|\alpha|} |\alpha|!^{s} \,e^{\tau \Re f(x,y)+s \tau^{1/s}},
\end{align*}
as we wished to prove.
\end{proof}

\section{Trace of ultradistributions}\label{AppendixB}

Let $u\in {\mathcal{D}}_s'(W)$ be an ultradistribution such that
\begin{align}\label{WFStraco}
WF_{s}(u) \cap \big\{(x,t, 0, \theta), (x,t) \in W, \theta \neq 0\big\}= \emptyset.  
\end{align}

We will see that condition \eqref{WFStraco} is enough to define the trace of $u$ at $t$, $\iota_t^{\ast}u\in {\mathcal{D}}_s'(B_R^{\R^{m}}(0))$,  as
\begin{align}\label{TRACE}
  (\iota_t^{\ast}u)( \varphi )&= \frac{1}{(2 \pi)^{N}} \int \mathcal{F}(\lambda u)(\sigma,\theta) e^{it \theta} \bigg(\int \varphi(x) e^{i x \sigma} \dd x\bigg) \dd \sigma \dd \theta, \quad \forall \varphi \in G^{s}_c(B_{R}^{\R^m}(0)),
\end{align}
where $\lambda\in G^{s}_c(W)$ is equal to $1$ in $V$ and $\mathcal{F}(\lambda u)$ stands for the Fourier transform of $\lambda u$. Note that, when $u\in\mathcal{D}'(W)$, this is the classical definition of trace of a distribution, see  \cite[Section 8]{Hor90a}.

\begin{Pro}\label{ProB1}
Let $u\in {\mathcal{D}}_s'(W)$ be an ultradistribution such that 
\eqref{WFStraco} is valid then $\iota_{t}^{\ast} u$ given by \eqref{TRACE} is in ${\mathcal{D}}_s'(B_{R}^{\R^{m}}(0))$. Moreover, for each fixed $\varphi \in G^{s}_c(B_R^{\R^{m}}(0))$ the function $B^{\Rn}_R(0) \ni t \mapsto \iota_{t}^{\ast}u(\varphi)$ is in $G^{s}(B^{\Rn}_R(0))$. 
\end{Pro}

\begin{proof}
Since $WF_{s}(u)$ and $\{(x,t, 0, \theta), (x,t) \in \supp \lambda, \theta \neq 0\}$ are disjoint closed cones, there is $\rho>0$ such that
\begin{align*}
WF_{s}(u) \cap \big\{(x,t, \sigma, \theta), (x,t) \in \supp \lambda,\theta \neq 0, \textrm{ and } |\sigma|\leq \rho |\theta|\big\}= \emptyset.
\end{align*}
Let $A_{1}= \{ (\sigma, \theta): |\sigma|\leq \rho|\theta|\}$ and $A_{2}= \RN\setminus A_{1}.$
Thus, 
\begin{equation}\label{traco11}
\del_{t}^{\beta}(\iota_{t}^{\ast} u)(\varphi)= I_1+ I_2
\end{equation}
where
\begin{align}\label{traco12}
  I_k= \frac{1}{(2\pi)^{N}}\int_{A_{k}} \mathcal{F}(\lambda u)(\sigma,\theta) (i\theta)^{\beta} e^{it \theta} \bigg(\int \varphi(x) e^{i x \sigma} \dd x\bigg) \dd \sigma \dd \theta, \quad k=1,2.
\end{align}
 Then,  there are $h>0$ and $C>0$ such that
\begin{align}\label{traco21}
|\mathcal{F}(\lambda u)(\sigma, \theta)|\leq C e^{-h |(\sigma, \theta)|^{1/s}}, 
\end{align}
for every $ (\sigma, \theta) \in A_{1}$ and for every $\epsilon>0$ there exist a positive constant $C_\epsilon$ such that
\begin{align}\label{estA2}
  |\mathcal{F}(\lambda u) (\sigma, \theta)|\leq C_\epsilon e^{\epsilon |(\sigma, \theta)|^{1/s}},  
\end{align}
for every  $(\sigma, \theta) \in A_{2}.$
Moving on, assume that $\varphi \in G^{s,r}_c(B_R^{\R^{m}}(0))$, therefore there is a constant $\tilde{C}$ depending only on $m$ and $R$ and $a$ depending only on $\rho,r$ and $s$ (see inequality~\eqref{DesigualdadePraFourier} below) such that
\begin{align}\label{Fourierestimate}
  |\mathcal{F}\varphi(\sigma)|&\leq \tilde{C} \|\varphi\|_{r, B_{R}^{\R^{m}}(0)}e^{-\frac{s}{r^{1/s}}|\sigma|^{1/s}}\nonumber \\
    &\leq \tilde{C} \|\varphi\|_{r, B_{R}^{\R^{m}}(0)}e^{-a |(\sigma,\theta)|^{1/s}}, 
\end{align}
for every  $(\sigma, \theta) \in A_{2}$.
In one hand, if we choose $\tilde{h}= h^{-s}(2s)^s $ it follows that 
\begin{align}\label{termI1}
  |I_1| &\leq \|\varphi\|_{r, B_{R}^{\R^{m}}(0)} \frac{C \tilde{C}}{(2\pi)^{N}} \int_{A_{1}} |\theta|^{|\beta|}   e^{-h |(\sigma, \theta)|^{1/s}}  \dd \sigma \dd \theta\nonumber\\
&\leq  \|\varphi\|_{r, B_{R}^{\R^{m}}(0)} \tilde{h}^{|\beta|} |\beta|!^s \frac{C \tilde{C}}{(2\pi)^{N}} \int_{A_{1}} \frac{|(\sigma,\theta)|^{|\beta|}}{\tilde{h}^{|\beta|}|\beta|!^{s}}   e^{-h |(\sigma, \theta)|^{1/s}}  \dd \sigma \dd \theta\nonumber\\&\leq  \|\varphi\|_{r, B_{R}^{\R^{m}}(0)} \tilde{h}^{|\beta|} |\beta|!^s \frac{C \tilde{C}}{(2\pi)^{N}} \int_{A_{1}}   e^{- \frac{h}{2} |(\sigma, \theta)|^{1/s}}  \dd \sigma \dd \theta.
\end{align}
On the other hand, we can choose $\epsilon<a/2$  and $\tilde{a}= a^{-s}(4s)^s$ to obtain that
\begin{align}\label{termI2}
  |I_2|&\leq \|\varphi\|_{r, B_{R}^{\R^{m}}(0)}\frac{C_\epsilon \tilde{C} }{(2\pi)^{N}} \int_{A_2} |\theta|^{|\beta|}  e^{-\frac{a}{2} |(\sigma, \theta)|^{1/s}} \dd \sigma \dd \theta\nonumber\\
&\leq \|\varphi\|_{r, B_{R}^{\R^{m}}(0)} \tilde{a}^{|\beta|}|\beta|!^s\frac{C_\epsilon \tilde{C} }{(2\pi)^{N}} \int_{A_2}   e^{-\frac{a}{4} |(\sigma, \theta)|^{1/s}} \dd \sigma \dd \theta\nonumber.
\end{align}

Taking $b= \max\{ \tilde{a}, \tilde{h}\}$ we conclude that there is $C'>0$ such that
then it holds 
\begin{equation}\label{the-end}
|\del_t^{\beta}(\iota_t^{\ast}u)(\varphi)|\leq C'\|\varphi\|_{r, B_{R}^{\R^{m}}(0)} b^{|\beta|}|\beta|!^s  .
\end{equation}
This means that, if $t\in B_{R}^{\R^{n}}(0)$ is fixed, then  $\del_t^{\beta}(\iota_t^{\ast}u)$ is a continuous linear functional in $ G^{s}_c(B_{R}^{\R^{m}}(0))$, i.e., $\del_t^{\beta}(\iota_t^{\ast}u) \in {\mathcal{D}}_s'(B_R^{\R^m}(0))$. Moreover if $\varphi$ is fixed, then \eqref{the-end} shows that $(\iota_t^{\ast} u)(\varphi) \in G^{s}(B_{R}^{\R^n}(0))$.
\end{proof}

\begin{Pro}\label{Leibiniz}
Let $u\in {\mathcal{D}}_s'(W)$ be an ultradistribution such that 
\eqref{WFStraco} is valid and $\psi\in G^s(W)$ then 
 the function $B^{\Rn}_R(0) \ni t \mapsto \iota_{t}^{\ast}u(\varphi(\cdot,t))$ is in $G^{s}(B^{\Rn}_R(0))$ and the following Leibniz formula holds
 \begin{equation}
 \del_t^{\beta}\big\{(\iota_t^{\ast}u)_x(\varphi(x,t))\big\} = \sum_{\alpha\le\beta} {\beta\choose\alpha}  \big(\del_t^{\alpha}(\iota_t^{\ast}u)\big)_x(\del_t^{\beta-\alpha}\varphi(x,t)).
 \end{equation}
\end{Pro}

\begin{proof}
It follows from the same arguments as in the proof of Proposition~\ref{ProB1}.
\end{proof}

Note that if $\varphi\in G^{s,r}_c(B_R^{\R^{m}}(0)),$ then
  \begin{align*}
    |\xi^{\alpha} \mathcal{F}\varphi(\xi)|&\leq \big| \int \big(D_{x}^{\alpha}\varphi(x)\big) e^{i x\xi} \dd x\big|\\
    & \leq \mu(B_{R}^{\R^{m}}(0)) \|\varphi\|_{r, B_{R}^{\R^{m}}(0)} r^{|\alpha|} |\alpha|!^s,
  \end{align*}
  where $\mu$ stands for the Lebesgue measure in $\R^{m}$. Moreover, since there exist a positive constant $C_m$ depending only of the dimension $m$ such that
  \begin{align*}
    |\mathcal{F}\varphi(\xi)| & \leq C_m \mu(B_{R}^{\R^{m}}(0)) \|\varphi\|_{r, B_{R}^{\R^{m}}(0)}\frac{ r^{|\alpha|} |\alpha|!^{s}}{ |\xi|^{|\alpha|}},\quad \xi\ne0
  \end{align*}
  holds for every $\alpha \in \Z_+^{m}$, we obtain
  \begin{align}\label{DesigualdadePraFourier}
    |\mathcal{F}\varphi(\xi)| & \leq C_m\mu(B_{R}^{\R^{m}}(0)) \|\varphi\|_{r, B_{R}^{\R^{m}}(0)} e^{- \frac{s}{r^{1/s}}|\xi|^{1/s}}.
  \end{align}


\subsection{Final Remark: definition of the restriction}

Observe that $\iota_t^{\ast} u$ in $V$ is independent of the choice of $\lambda$ in the following way: for any $\varphi\in G^s_c(B_{R}^{\Rm}(0))$ and any $t \in B_{R}^{\R^n}(0)$,  function $\iota_t^{\ast} u(\varphi)$  does not dependent of the choice of $\lambda$ in $G^{s}_c(W)$ as long as $\lambda=1$ in $V$. To see use that if $\psi \in G^{s}_c(B_{R}^{\R^{n}}(0))$, then
\begin{align*}
  u(\psi\otimes \varphi)&= (\lambda u) (\psi\otimes \varphi)\\
                        &=\frac{1}{(2\pi)^{N}}\int \mathcal{F}(\lambda u)(\sigma,\theta) \bigg(\int  e^{it \theta}\Psi(t) \dd t\bigg)  \bigg(\int \varphi(x) e^{i x \sigma} \dd x\bigg) \dd \sigma \dd \theta\\
  &= \int (\iota^{\ast}_t u)(\varphi) \psi(t)\dd t.
\end{align*}
The equality above implies that the function $t\mapsto \iota_t^{\ast}u(\varphi)$ is uniquely determined in $B_R^{\Rn}(0)$.

\bibliographystyle{alpha}
\bibliography{Bibliografia}

\end{document}